\documentclass[12pt,reqno]{amsart}
\usepackage{epsfig,listings}
\usepackage{verbatim}
\usepackage{amssymb,amsmath}
\usepackage{amsfonts,mathrsfs}
\usepackage{amsbsy}
\usepackage{graphicx}
\usepackage{subfigure}
\usepackage{caption}
\usepackage[usenames]{color} 

\def\type{\mathop{\rm Type}\nolimits}

\newtheorem{thm}{Theorem}[section]
\newtheorem{lem}[thm]{Lemma}
\newtheorem{prop}[thm]{Proposition}

\theoremstyle{definition}
\newtheorem{defn}[thm]{Definition}
\newtheorem{exam}{Example}[section]
\newtheorem{rem}[thm]{Remark}

\usepackage{color}
\usepackage{setspace}

\usepackage [colorlinks, citecolor=blue, anchorcolor=blue, linkcolor=blue ]{hyperref}
\usepackage{framed} %
\usepackage{diagbox}%
\usepackage[justification=centering]{caption}

\begin{document}
\title{ Optimal embedding of hypercube into cylinder }

\author{Zhiyi Tang}
\address{
	School of  Computer Science and Technology,
	Beijing Institute of Technology,
	Beijing 100081, P. R. China.}
\email{zytang@bit.edu.cn}
	
\begin{abstract}
	
	We study  the problem of Embedding Wirelength of   $n$-dimensional Hypercube $Q_n$ into Cylinder $C_{2^{n_1}}\times P_{2^{n_2}}$, where $n_1+ n_2=n$, called EWHC.
	We show that such wirelength  corresponding to Gray code embedding is $2^{n_2}(3\cdot 2^{2n_1-3}-2^{n_1-1})
	+2^{n_1} (2^{2n_2-1}-2^{n_2-1})$. 
	In addition, we prove that Gray code embedding  is an optimal strategy of EWHC.\\ \par\

	
\end{abstract}
	
	\maketitle	
	
\section{Introduction}
Graph embedding is important in parrallel algorithm,parrallel computer, or multiprocessor systems.
The hypercube is one of the most popular, versatile and efficient topological structures of interconnection networks \cite{Xu}. 
Several  studies related to the embedding problem have been performed,such as 
 hypercube into  path \cite{H1964}, hypercube into grid \cite{BCHRS,MRRM}, hypercube into complete binary tree \cite{TREE} or  Caterpillar CAT \cite{PM}. 
 Even though there are numerous results and discussions on the wirelength problem, most of them are basically focus on the host graphs without cycles.
As for the host graphs with cycles, the embedding problems are more complex. The major challenge lies in finding a partition of the edge set of the host graph and minimizing the induced total wirelength.

 Manuel et al (\cite{PM}) and Liu et al (\cite{J}) considered some sepcial cases of EWHC, where only four or eight nodes are permitted in each loop of cylinder. The minimum EWHC for general cases has been proposed as a conjecture in \cite{PM}. \cite{RRPR} claimed to prove EWHC conjecture, but \cite{AS2015} pointed out that there are logical flaws in the proof of \cite{RRPR}. Hence the proof of EWHC conjecture remains open\cite{AS2015}.  
Recently, Liu  and Tang (\cite{Liu2021}) give a rigorous proof of the conjecture of circular wirelength of hypercubes.
The present paper develop the method in \cite{Liu2021} and try to minimize EWHC by Gray code embedding.

Let $G$ and $H$ be finite undirected graphs with $n$ vertices.
$V(G)$ and $V(H)$ denote the vertex sets of $G$ and $H$, respectively.
$E(G)$ and $E(H)$ denote the edge sets of $G$ and $H$, respectively.
An embedding (\cite{BCHRS}) $f$ of $G$ into $H$ is defined as follows:

(i). $f$ is a bijective mapping from $V(G)$ to $V(H)$;

(ii). $f$ is a injective mapping from $E(G)$ to $2^{E(H)}$,
which assigns for each edge $e = \{u, v\}\in E(G)$ a  path $P_f(u,v)$ in $H$ between $f(u)$ and $f(v)$.
	
In this paper, we always take $P_f(u,v)$ to be a shortest path from $f(u)$	to $f(v)$ in $H$.
The distance between $f(u)$ and $f(v)$ in $H$,  denoted by $d_H(f(u),f(v))$, is the length of the path $P_f(u,v)$. 	
Based on the embedding $f$, the wirelength of $G$ into $H$ is given by
\begin{equation*}
	WL(G,H;f)=\sum_{\{u,v\}\in E(G)}d_H(f(u),f(v)).
\end{equation*}
In addition, taking over all embeddings $f$, the minimum wirelength of $G$ into $H$ is defined as
	$$WL(G,H)=\min\limits_{f} WL(G,H;f).$$

In the paper, we give an exact formula of minimum EWHC
in the following theorem.
\begin{thm}\label{cpthm}
	For any $n\ge 2$, $n_1,\ n_2\ge1,\ n_1+ n_2=n$.
	\begin{equation}\label{cpwl}
		WL(Q_n,C_{2^{n_1}}\times P_{2^{n_2}})=2^{n_2}(3\cdot 2^{2n_1-3}-2^{n_1-1})
		+2^{n_1} (2^{2n_2-1}-2^{n_2-1}).
	\end{equation}
Moreover, Gray code embedding is an optimal embedding of EWHC.
\end{thm}

If $n_1=0,n_2=n$, then the cylinder degenerates into a path, which is solved by \cite{H}.
 If $n_1=n,n_2=0$, then the cylinder degenerates into a cycle, which is solved by  \cite{Guu,Liu2021}, the theorem still holds. 
 
Manuel et al (\cite{PM}) prove \eqref{cpwl} in the case of embedding $n_1=2$, where either lexicographical order or Gray code embedding minimizes the wirelength.
Liu et al (\cite{J}) show \eqref{cpwl} in the case of embedding $n_1=3$, where Gray code embedding minimizes the wirelength. Our work is to devote to  the general cases.

The rest of this paper is organized as follows. Section 2 starts with two equivalent ways to compute the wirelength of an embedding, ECP and EIP, which are rather straightforward in Lemma \ref{wl;theta}.
In Section 3, we propose a partition of the edge set of cylinder and the derivate wirelength formula \eqref{wlsum} is obtained corresponding to this partition with EIP. 
In Section 4, we define Gray code embedding of hypercube into a cylinder and  compute the exact  wirelength \eqref{wlgray} based on Gray code embedding.
In Section 5, we show that Gray code embedding is an optimal strategy of EWHC, which is turned into Theorem \ref{cpthm2} and Theorem \ref{cpthm1}.
Concluding remarks are given in Section 6.

\section{Preliminaries}	
Computations of the embedding wirelength between different graphs have been widely analyzed. In this section, we  list two major approaches: edge congestion problem(ECP) and edge isoperimetric problem(EIP), and focus on EIP for hypercubes later.

\subsection{Edge congestion problem}\label{ecp}\

ECP is one of the most popular tools studied in the computation of wirelength of different embeddings. With ECP, one could treat the wirelength of an embedding as a  problem of  some congestion summation through all  edges of $H$(such as \cite{BCHRS,J,MRRM,PM}). 
Here we only state some necessary definitions and notations, more results about ECP are referred to \cite{BCHRS,MRRM}.
\begin{defn}[Def.2 of \cite{TCT}]
	The {\it edge congestion} of an edge $e'\in E(H)$ when using an embedding $f$, denoted by $EC_f(e')$, is defined as the number of $e=\{u,v\}\in E(G)$ such that $e'$ is incident on $P_f(u,v)$. In other words,
	$$EC_f(e')=\#\{\{u,v\}\in E(G): e'\in P_f(u,v)\}.$$
\end{defn}
This can be extended any edge subset $Y\subset E(H)$, define
$$EC_f(Y)=\sum_{e'\in Y}EC_f(e').$$
For an embedding $f$, we define an index function $\delta(e',P_f(u,v))$ to be one if $e'\in P_f(u,v)$ and zero  otherwise for each $\{u,v\}\in E(G),\ e'\in E(H)$. Then
$EC_f(e')$, $d_H(f(u),f(v))$ are equivalent to $\sum_{\{u,v\}\in E(G)} \delta(e',P_f(u,v))$, 
$\sum_{e'\in E(H)} \delta(e',P_f(u,v))$, respectively.
Consequently, the embedding wirelength of $G$ into $H$ is equivalent to $\sum_{e'\in E(H)}EC_f(e')$. In other words,
$$WL(G,H;f)=\sum_{e'\in E(H)}EC_f(e').$$

Clearly, take any partition $(Y_i)_{i=1}^{m}$ of $E(H)$, then
$$WL(G,H;f)=\sum_{i=1}^{m} EC_f(Y_i).$$

In particular, $G$ is a regular graph. 
Based on  Congestion Lemma of \cite{MRRM}, for each $1\le i \le m$, let  $Y_i$ be an edge cut of $H$ such that the removal of edges of $Y_i$ leaves $H$ into 2 components. One induced vertex set is denoted by $\beta_i$, and let $\alpha_i=f^{-1}(\beta_i)$, edge set $R_i=\{\{u,v\}\in E(G): u\in \alpha_i,v\notin \alpha_i\}$.
Suppose $Y_i$ satisfies the followsing conditions:

\textbf{(A1)}.For each $\{u,v\}\in R_i$, $P_f(u,v)$ has exactly one edge in $Y_i$;

\textbf{(A2)}.For each $\{u,v\}\in E(G)\backslash R_i$, $P_f(u,v)$ has no edge in $Y_i$.

Then 
\begin{equation*}\label{wl;R}
	WL(G,H;f)=\sum_{i=1}^{m} |R_i|.
\end{equation*}
Since conditions (A1)(A2) imply that 
$\sum_{e'\in Y_i} \delta(e',P_f(u,v))$ is one if $\{u,v\}\in R_i$  and zero otherwise for each $\{u,v\}\in E(G)$.
What's more, for each $1\le i \le m$,
$$EC_f(Y_i)=\sum_{e'\in Y_i}\sum_{\{u,v\}\in E(G)} \delta(e',P_f(u,v))=\sum_{\{u,v\}\in R_i}\sum_{e'\in Y_i}\delta(e',P_f(u,v))=|R_i|.$$

It is worth mentioning that \cite{PM}  studied EWHC of the case $n_1=2$ by introducing some partition of edge set $E(C_{2^{2}}\times P_{2^{n_2}})$, then minimized each item with the help of ECP. 
\cite{J} extended to the case $n_1=3$ with ECP. 
But for $n_1>3$, no one  find an appropriate partition to  minimize each item with ECP until now.  The purpose of the paper is aim to present an alternative partition of $E(C_{2^{n_1}}\times P_{2^{n_2}})$, then to apply EIP to solve EWHC for more general cases.

\subsection{Edge isoperimetric problem}\

 For any graph $G$, and $S\subset V(G)$, define $$\theta(S)=\#\{\{u,v\}\in E(G):u\in S, v\notin S \}.$$ 
 
The above results can be summarized in the following lemma.
\begin{lem}\label{wl;theta}
	Let $G$ be a regular graph and $f$ be an embedding of $G$ into $H$.
	Let $(Y_i)_{i=1}^{m}$ be a  partition  of $E(H)$.
	If each $Y_i$ is an edge cut of $H$ such that 
	$Y_i$ disconnects $H$ into 2 components, induced  vertex set  $\beta_i$ and $\beta_i^c$. 
	Let $\alpha_i=f^{-1}(\beta_i)$, edge set $R_i=\{\{u,v\}\in E(G): u\in \alpha_i,v\notin \alpha_i\}$,
	and $Y_i$ satisfies  conditions(A1)(A2). Then
	\begin{equation*}
	WL(G,H;f)=\sum_{i=1}^{m} \theta(\alpha_i).
	\end{equation*}
\end{lem}

Consider $G$ to be a hypercube of dimension $n$, denoted by $Q_n$. The vertex set of $Q_n$, $V(Q_n)=\{0,1\}^n$, are all $n$-tuples over two letters alphabet $\{0,1\}$.
For any $u,v\in V(Q_n)$, $\{u,v\}\in E(Q_n)$,  if and only if $u,v$ differ in exactly one coordinate.

For any $S\subset V(Q_n)$, \cite{H} denote
$$\theta(n,S)=\#\{ e\in E(Q_n) :\ e=\{v,w\},v\in S,w \notin S  \}.$$
What's more, for each $0\le k\le 2^{n}$, denote
$$\theta(n,k)=\min\{\theta(n,S) :\ S\subset V(Q_n), |S|=k\}.$$
What kinds of set $S\subset V(Q_n)$ with $|S|=k$ satisfies $\theta(n, S)=\theta(n,k)$, is called {\it the  discrete isoperimetric problem (EIP) for hypercubes}.
Harper(\cite{H1964,H}) developed EIP theory and did a series of good jobs on hypercubes embeddings. He proved that
\begin{lem}[\cite{H}]\label{Hcubal}
	For $S \subset V(Q_n)$ with $|S|=k$,	
	$\theta(n, S)=\theta(n,k)$ if and only if $S$ is a cubal. 
\end{lem}

\begin{defn}[\cite{H,Liu2021}]
	Take $S \subset V(Q_n)$ with $|S|=k $, where we write
	$$k=\sum_{i=1}^{N} 2^{c_i},\ N\ge 1,\ 0\le c_1<c_2<\cdots<c_N=\left \lfloor  \log_2 k\right \rfloor.$$
	If $S$ is a disjoint union of  $c_i$-subcubes,$i=1,\cdots,N$,
	such that each $c_i$-subcube lies in a neighbor of every $c_j$-subcube for $j<i$, then $S$ is called a {\it $k$-cubal}.
	Letting $E(k)=|E(S)|$ if $S$ is a cubal with $k$ vertices, then
	\begin{equation}\label{E(k)}
	E(k)=\sum_{i=1}^{N}(N-i)2^{c_i}+c_i 2^{c_i-1}.
	\end{equation}
\end{defn}
Consequently,
\begin{equation}\label{theta(k)}
\theta(n,k)=n\cdot k-2E(k).
\end{equation}

As a corollary, we have
\begin{equation}\label{recurred}
	\theta(n,k)=\theta(n-1,k)+k,\quad 0\le k\le 2^{n-1},
\end{equation}
and
\begin{equation}\label{double}
\theta(n+1,2k)=2\theta(n,k),\quad 0\le k \le 2^n.
\end{equation}

Next, we introduce another powful indicator,{\it Type}.
\begin{defn}[\cite{Guu,Liu2021}]
	Given any $S\subset V(Q_n)$, define
	$$\type(S):=\min\limits_{H\in\mathscr{H}_n}|S\cap H|,$$
	where $\mathscr{H}_n$ is the set of $2n$ half planes of $Q_n$.
	For any $n>0$, $0\le k\le 2^n$, $t\le k/2$, we denote,
	$$\theta(n,k,t)=\min\{\theta(n,S):\ S\subset V(Q_n),\ |S|=k,\ \type(S)=t\}.$$
	We call a subset $S$ with $|S|=2^{n-1}$ is of small type, if $0\le \type(S)\le 2^{n-3}$.
	Otherwise, we call it is of big type.
\end{defn}
If $S$ is of small type, then we can directly compute its $\theta$ value by formula
\begin{equation}\label{small}
\theta(n,2^{n-1},t)=2\theta(n-2,t)+2^{n-1},\ 0\le t \le 2^{n-3}.
\end{equation}
What's more, for $0\le t\le 2^{n-4}$, by \eqref{small} and \eqref{double}, we have
\begin{equation}\label{double2}
	\theta(n,2^{n-1},2t)=2\theta(n-1,2^{n-2},t).
\end{equation}

If $S$ is of big type, then we can estimate the lower bound of its $\theta$ value by relation 
\begin{equation}\label{big}
\theta(n,2^{n-1},t)\ge \theta(n,2^{n-1},2^{n-3}),\quad 2^{n-3} \le t \le 2^{n-2}.
\end{equation}
The details of \eqref{small} and \eqref{big} are referred to \cite{Liu2021}.

\section{A Partition of $E(C_{2^{n_1}}\times P_{2^{n_2}})$}\

In this section, we present notations of the cylinder, including its vertex set and edge set. Then we apply EIP to  compute EWHC corresponding to a partition of the edge set of cylinder.

 In convention, denote $1-$dimensional path and circle with $n\ge 2$ vertices by $P_n$ and $C_n$, respectively. 
The cylinder $C_{2^{n_1}}\times P_{2^{n_2}}$ is  a Cartesian product of a circle $C_{2^{n_1}}$ with a path $P_{2^{n_2}}$. Also, the cylinder $C_{2^{n_1}}\times P_{2^{n_2}}$ is seen as the grid $P_{2^{n_1}}\times P_{2^{n_2}}$ with a wraparound edge in each column.

It is nature to label the vertex set of $C_{2^{n_1}}\times P_{2^{n_2}}$ with two coordinates, i.e.,
$V=\{(x_1,x_2): 1\le x_i\le 2^{n_i},\ i=1,2\}$.
In a different way, the vertex set of $C_{2^{n_1}}\times P_{2^{n_2}}$ can be  labelled by a series of  nature numbers (one coordinate) in the following arrangement. (see Fig.\ref{c4p8})
$$\begin{array}{lllll}
&1              &2           &\cdots    &2^{n_2}\\
&2\cdot2^{n_2}  &2\cdot2^{n_2}-1      &\cdots    &2^{n_2}+1\\
&\cdots        &\cdots       &\cdots    &\cdots \\
&(2^{n_1}-2)\cdot2^{n_2}+1  &(2^{n_1}-2)\cdot2^{n_2}+2 &\cdots    &(2^{n_1}-1)\cdot2^{n_2}\\
&2^{n_1}\cdot2^{n_2}      &2^{n_1}\cdot2^{n_2}-1      &\cdots    &(2^{n_1}-1)\cdot2^{n_2}+1 
\end{array}$$
\begin{figure}[htbp]
\begin{minipage}[t]{0.45\linewidth}
	\centering
	\includegraphics[width=2.5in]{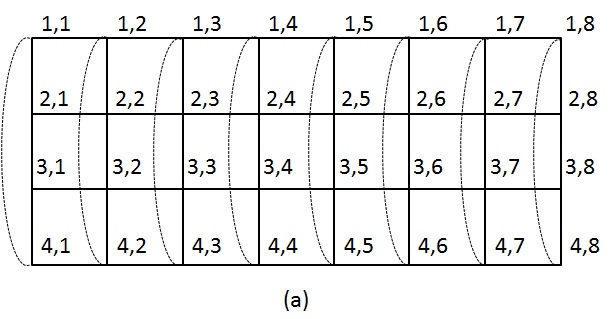}
\end{minipage}%
\begin{minipage}[t]{0.45\linewidth}
	\centering
	\includegraphics[width=2.5in]{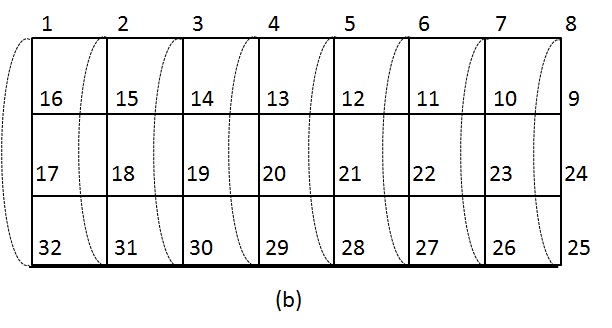}
\end{minipage}%
\caption{cylinder $C_{2^{2}}\times P_{2^{3}}$}
\label{c4p8}
\end{figure}

For any $n\ge 2$, $n_1,\ n_2\ge1,\ n_1+ n_2=n$.
Based on this labelling,  an
embedding of hypercube $Q_n$ into cylinder $C_{2^{n_1}}\times P_{2^{n_2}}$ is a map $f:\{0,1\}^n\rightarrow \{1,2,\cdots,2^n\}$.
For $1\le i \le 2^{n_1},\ 1\le j \le 2^{n_2}$, letting
$$a_{i,j}=\left\{
\begin{array}{cl}
(i-1)2^{n_2}+j,&\mbox{if }\ i\ \ \mbox{is\ odd},\\
i\cdot2^{n_2}+1-j,&\mbox{if }\ i\ \ \mbox{is\ even},\\
\end{array}
\right.$$
then the edge set of $C_{2^{n_1}}\times P_{2^{n_2}}$, 
$$\begin{array}{cl}
E(C_{2^{n_1}}\times P_{2^{n_2}})=&
\{
\{a_{i,j},a_{i,j+1}\}: 1\le i\le 2^{n_1}, 1\le j< 2^{n_2} \} \bigcup \\
&\{
\{a_{i,j},a_{i+1,j}\}: 1\le i< 2^{n_1}, 1\le j\le 2^{n_2} \}
\bigcup \\
&\{
\{a_{2^{n_1},j},a_{1,j}\}: 1\le j\le 2^{n_2}\}.
\end{array}
$$ 

Letting $$X_1=\{
\{a_{2^{n_1},j},a_{1,j}\},
\{a_{2^{n_1-1},j},a_{2^{n_1-1}+1,j} \}
:1\le j\le2^{n_2}\},$$ 
for $1< i\le2^{n_1-1}$,
 $$X_i=\{
 \{a_{i-1,j},a_{i,j}\},
 \{a_{i-1,j}+2^{n-1},a_{i,j}+2^{n-1}\}
 :1\le j\le2^{n_2}\},$$
  and for $1\le j<2^{n_2}$,
 $$Y_j=\{
 \{a_{i,j},a_{i,j+1}\}
 :1\le i\le 2^{n_1}\}.$$
 
Observe that $X_i$ disconnects $C_{2^{n_1}}\times P_{2^{n_2}}$ into two components, where the induced vertex set 
\begin{equation}\label{Ai}
	A_i=\{(i-1)2^{n_2}+k:1\le k\le 2^{n-1}\}, \quad 1\le i \le 2^{n_1-1},
\end{equation} 
and
$Y_j$ disconnects $C_{2^{n_1}}\times P_{2^{n_2}}$ into two components,  where the induced vertex set
\begin{equation}\label{Bj}
	B_j=\{a_{i,k}:1\le i\le 2^{n_1}, 1\le k\le j\},\quad 1\le j < 2^{n_2}.
\end{equation} 
It is direct to check that  $\{X_1,\cdots,X_{2^{n_1-1}},Y_1,\cdots,Y_{2^{n_2}-1}\}$ is a partition of $E(C_{2^{n_1}}\times P_{2^{n_2}})$ and  satisfies conditions(A1) (A2).(See Fig. \ref{c8p4}) 
\begin{figure}[htbp]
	\begin{minipage}[t]{0.45\linewidth}
		\centering
		\includegraphics[width=1.5in]{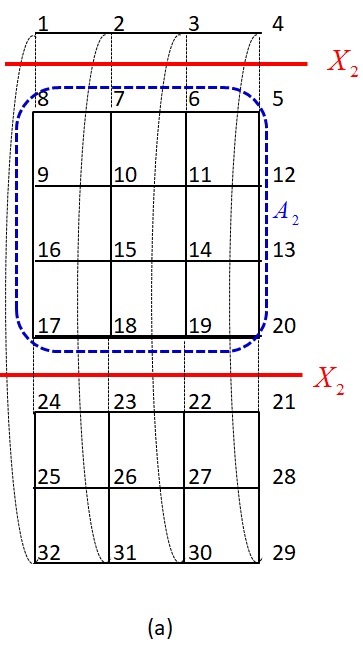}
	\end{minipage}%
	\begin{minipage}[t]{0.45\linewidth}
		\centering
		\includegraphics[width=1.5in]{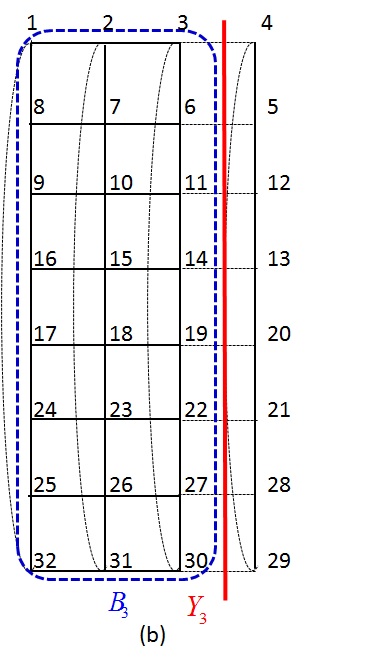}
	\end{minipage}%
	\caption{
	  (a) edge cut $X_2$ disconnects $C_{2^3}\times P_{2^2}$ into two components,where the induced vertex set is $A_2$.
	  (b) edge cut $Y_3$ disconnects $C_{2^3}\times P_{2^2}$ into two components,where the induced vertex set is $B_3$.
		}
	\label{c8p4}
\end{figure}

In the rest of paper, let $f$ be an
 embedding of hypercube $Q_n$ into cylinder $C_{2^{n_1}}\times P_{2^{n_2}}$ with $n_1+ n_2=n$, $A_i$ and $B_j$ are defined in \eqref{Ai} and \eqref{Bj}, respectively. By Lemma \ref{wl;theta}, the wirelength  
\begin{equation}\label{wlsum}
	WL(Q_n,C_{2^{n_1}}\times P_{2^{n_2}};f)=
	\sum_{i=1}^{2^{n_1-1}}\theta(n,f^{-1}(A_i))+\sum_{j=1}^{2^{n_2}-1}\theta(n,f^{-1}(B_j)).
\end{equation}

 
\section{ Gray code embedding of $Q_n$ into $C_{2^{n_1}}\times P_{2^{n_2}}$}\label{sec:gray}\

In this section, we study an embedding of $Q_n$ into $C_{2^{n_1}}\times P_{2^{n_2}}$  that is based on Gray code. 
\begin{defn}[\cite{Liu2021}]
	A {\it Gray code} is an ordering of $2^n$ binary numbers such that only one bit changes from one entry to the next.
	For $n>0$, define a Gray code map $\xi_n:\{0,1\}^n\rightarrow \{1,2,\cdots,2^n\}$.
	Take any  $u=u_1\cdots u_n\in\{0,1\}^n$.
	Let ${u}'={u_1}'\cdots {u_n}'\in\{0,1\}^n$ be defined as, for any $1\le i\le n$,
	${u_i}'=0$, if $u_1+\cdots+u_i$ is even; ${u_i}'=1$, otherwise. Then
	$\xi_n(u)=({u}')_2+1,$
	where $({u}')_2$ is the integer with binary expansion ${u}'$.
	For example: $\xi_5(01101)=(01001)_2+1=9+1=10$.
\end{defn}
	Clearly, $V(Q_n)=\{0,1\}^n$, $V(C_{2^{n_1}}\times P_{2^{n_2}})=\{1,2,\cdots,2^n\}$ for $n_1+ n_2=n$. Thus,
	$\xi_n$ is also identified as Gray code embedding of $Q_n$ into $C_{2^{n_1}}\times P_{2^{n_2}}$, where $n_1+ n_2=n$.

For $1\le i\le 2^{n-1}$, letting 
\begin{equation}\label{Si}
	S_i=\{i,i+1,\cdots,i+2^{n-1}-1\},
\end{equation}
and  
$G_i=\xi_n^{-1}(S_i),\ \type(G_i)=g_i^n$.
Notice that $G_i, \type(G_i)$ are same as Section 2.4 in \cite{Liu2021}. 
Moreover, for $1\le i\le 2^{n-1}$,
\begin{equation*}\label{gi}
	g_i^{n}=\left\{
	\begin{array}{cl}
		i-1,&\mbox{if } 1\le i\le 2^{n-3},\\
		2^{n-2}-(i-1),&\mbox{if } 2^{n-3}<i\le 2^{n-2},\\
		g_{i-2^{n-2}}^{n},&\mbox{if } 2^{n-2}< i\le 2^{n-1}.\\
	\end{array}
	\right.
\end{equation*}
 and $\theta(n,G_i)=\theta(n,2^{n-1},g_i^n)$.
 
Now, consider the embedding of $Q_n$ into $C_{2^{n_1}}\times P_{2^{n_2}}$ for $n_1+n_2=n$.
Observe that for $1\le i \le 2^{n_1-1}$, $A_i=S_{1+2^{n_2}(i-1)}$, and
\begin{equation*}
	\xi_n^{-1}(A_i)=G_{1+2^{n_2}(i-1)},
\end{equation*}
 \begin{equation}\label{graytype}
 	\type(\xi_n^{-1}(A_i))=g_{1+2^{n_2}(i-1)}^n=2^{n_2}g_i^{n_1}.
 \end{equation}
Together with \eqref{double2} repeatedly, we have
\begin{equation}\label{gi-n2}
\theta(n,\xi_n^{-1}(A_i))
=\theta(n,2^{n-1},2^{n_2}g_i^{n_1})
=2^{n_2}\theta(n_1,2^{n_1-1},g_i^{n_1}).
\end{equation}
Moreover, by the  circular wirelength formula\cite{Liu2021}
$$
\sum\limits_{i=1}^{2^{n-1}}\theta(n,2^{n-1},g_i^n)
=3\cdot 2^{2n-3}-2^{n-1}, $$ 
so, we have
\begin{equation}\label{cpwl1}
	\sum_{i=1}^{2^{n_1-1}} \theta(n,\xi_n^{-1}(A_i))=
	2^{n_2}(3\cdot 2^{2n_1-3}-2^{n_1-1}).
\end{equation}
\begin{figure}[htbp]
	\centering
	\includegraphics[width=2.5in]{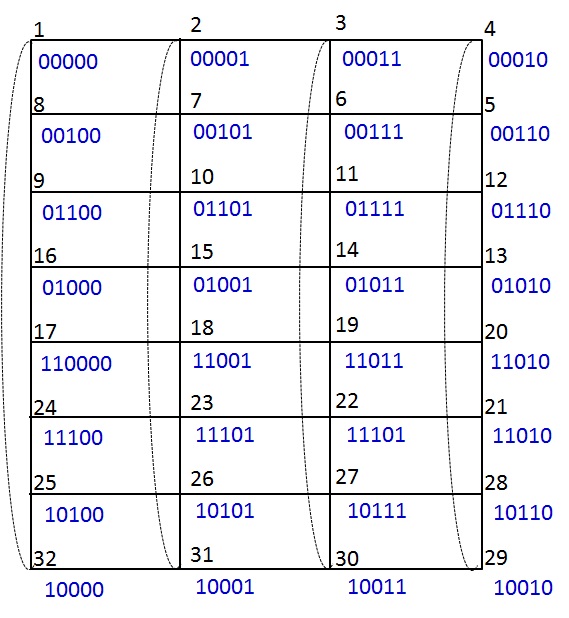}
	\caption{Gray code in $C_{2^{3}}\times P_{2^{2}}$}
	\label{gray}
\end{figure}

In addition, by the recurrent structure of Gray code embedding, \ it is direct to check that for each $1\le j \le 2^{n_2}-1$, 
$\xi_n^{-1}(\{a_{i,j}: 1\le i \le 2^{n_1} \})$ is a $n_1-$dimensional subcube.(see Fig.\ref{gray}).
Lemma 4.2 of \cite{J} shows that for  each $1\le i \le 2^{n_1}$, $1\le j \le 2^{n_2}-1$, 
$\xi_n^{-1}(\{a_{i,k}: 1\le k \le j \})$ is a cubal of $Q_{n}$. 
Notice that for each $1\le j \le 2^{n_2}-1$,  $\xi_n^{-1}(B_{j})$ is a Cartesian product of a subcube of order $n_1$ with a cubal.
So,by Lemma 4.1(ii) of \cite{J}, for each $1\le j \le 2^{n_2}-1$, $\xi_n^{-1}(B_{j})$ is a cubal of $Q_{n}$.  
Hence, by Lemma \ref{Hcubal}, for  $1\le j \le 2^{n_2}-1$, we have 
$$\theta(n,\xi_n^{-1}(B_{j}))
=\theta(n,|B_{j}|)=\theta(n,j\cdot 2^{n_1}).$$
Moreover, by \eqref{double} repeatedly, we have $\theta(n,\xi_n^{-1}(B_{j}))=2^{n_1}\theta(n_2,j )$.
Together with the wirelength formula of hypercube into a path\cite{H1964,H}
$$	\sum\limits_{i=1}^{2^n-1}\theta(n,i)
=2^{2n-1}-2^{n-1},$$
so, we have
\begin{equation}\label{cpwl2}
	\sum_{j=1}^{2^{n_2}-1}\theta(n,\xi_n^{-1}(B_{j}))
	=2^{n_1}(2^{2n_2-1}-2^{n_2-1}).
\end{equation} 

At last, take \eqref{cpwl1} and \eqref{cpwl2} into \eqref{wlsum}, we have the wirelength of Gray code embedding from $Q_n$ into $C_{2^{n_1}}\times P_{2^{n_2}}$.
\begin{equation}\label{wlgray}
	WL(Q_n,C_{2^{n_1}}\times P_{2^{n_2}};\xi_n)=
	2^{n_2}(3\cdot 2^{2n_1-3}-2^{n_1-1})+
	2^{n_1}(2^{2n_2-1}-2^{n_2-1}).
\end{equation}
\section{Optimal embedding of  $Q_n$ into $C_{n_1}\times P_{n_2}$}\

In this section, we explain that the Gray code embedding  is an optimal strategy of EWHC.
Section \ref{sec:gray} has illustrated that 
 for each $1\le j \le 2^{n_2}-1$,  $\xi_n^{-1}(B_{j})$ is a cubal.
 This implies that
 Gray code embedding minimizes each $\theta(n,f^{-1}(B_{j}))$, consequently,
 \begin{thm}\label{cpthm2}
 	For any $n\ge 2$, $n_1,n_2\ge1,\ n_1+n_2=n$,  Gray code embedding minimizes $\sum_{j=1}^{2^{n_2}-1} \theta(n,f^{-1}(B_{j}))$.
 \end{thm}
 Together with \eqref{wlsum}, we only need to prove that
\begin{thm}\label{cpthm1}
For any $n\ge 2$, $n_1,n_2\ge1,\ n_1+n_2=n$,  Gray code embedding minimizes $\sum_{i=1}^{2^{n_1-1}} \theta(n,f^{-1}(A_i))$.
\end{thm}

Notice that for each $1\le i \le 2^{n_1-1}$, 
$f^{-1}(A_i)$ consists of $2^{n-1}$ vertices, hence 
$\sum_{i=1}^{2^{n_1-1}} \theta(n,f^{-1}(A_i))$ is largely related to its each type.
In the following, we firstly summarize the derivative properties of these types, then finish the proof of Theorem \ref{cpthm1}.

\subsection{Type sequence of $f^{-1}(A_i)$}\

In this subsection, we follow the results in 
\cite{Guu,Liu2021} to discuss the properties of the induced type sequence corresponding to an embedding $f$. 
For $1\le i \le 2^{n-1}$,  letting 
$P_i=f^{-1}(S_i), \type(P_i)=t_i$, where $S_i$ is defined in \eqref{Si}, then
\begin{lem}[\cite{Guu}, Prop.2.6 of \cite{Liu2021} ]\label{prop8}
	The sequence $(t_i)_{i=1}^{2^{n-1}}$ satisfies	
	\begin{itemize}
		\item[(i)] circular continuous,i.e.,for any $1\le i<2^{n-1}$, $|t_i-t_{i+1}|\le1$;	
		\item[(ii)] there exist at least two $1\le i\le2^{n-1}$ such that $t_i \ge 2^{n-3}$.
	\end{itemize}
\end{lem}

Now, we suggest some new notations to describe $(\type(f^{-1}(A_i)))_{i=1}^{2^{n_1-1}}$.
\begin{defn}
For $m,n>0$, an integer sequence $(x_i)_{i=1}^n$  is called an {\it $m-$level circular continuous} sequence, if  for $1\le i <n,\ |x_i-x_{i+1}|\le 2^{m}$, and $|x_n-x_1|\le 2^{m}$.
\end{defn}

For any $n\ge 2$, $n_1,n_2\ge1,\ n_1+n_2=n$.
It is direct to see that for $1\le i \le 2^{n_1-1}$,
$$ f^{-1}(A_i)=P_{1+2^{n_2}(i-1)},\ \type(f^{-1}(A_i))=t_{1+2^{n_2}(i-1)}.$$
For brevity, denote $t_{1+2^{n_2}(i-1)}$ by $s_i$. 
Clearly, $(s_i)_{i=1}^{2^{n_1-1}}$ is a subsequence of sequence $(t_i)_{i=1}^{2^{n-1}}$ and possesses similar properties with $(t_i)_{i=1}^{2^{n-1}}$.

\begin{prop}\label{tyseq}
	For any $n\ge 2$, $n_1,n_2\ge1,\ n_1+n_2=n$. The sequence $(s_i)_{i=1}^{2^{n_1-1}}$ satisfies	
	\begin{itemize}
		\item[\textbf{(C1)}] $(s_i)_{i=1}^{2^{n_1-1}}$ is $n_2-$level circular continuous;	
		\item[\textbf{(C2)}] there exist at least two $1\le i \le 2^{n_1-1}$ such that $s_i\ge2^{n-3}-2^{n_2-1}$.
	\end{itemize}
\end{prop}
\begin{proof}
(1). Without loose of generality(WLOG),we only prove $s_1,\ s_2$, the others are same. It is direct to verify that
	$$|s_1-s_2|=|t_{1}-t_{1+2^{n_2}}|
	\le|t_{1}-t_{2}|+|t_{2}-t_{3}|+\cdots+|t_{2^{n_2}}-t_{1+2^{n_2}}|
	\le2^{n_2},$$
where the last inequality is due to Lemma \ref{prop8} (i). So, (C1) is proved.
	
(2). By Lemma \ref{prop8} (ii), assume $t_k\ge 2^{n-3}$ for some $1\le k\le 2^{n-1}$. Due to Lemma \ref{prop8} (i),  $t_{k+j}\ge2^{n-3}-2^{n_2-1}$ for each $-2^{n_2-1}\le j\le 2^{n_2-1}$. 
Notice that $(s_i)_{i=1}^{2^{n_1-1}}$ is a subsequence of sequence $(t_i)_{i=1}^{2^{n-1}}$ and $n_2-$level circular continuous. Then, there must exist some  $1\le i_1\le 2^{n_1-1}$,$-2^{n_2-1}\le j_1\le 2^{n_2-1}$, such that $s_{i_1}=t_{k+j_1}\ge2^{n-3}-2^{n_2-1}$.  Similarly, we can find another $s_{i_2}\ge2^{n-3}-2^{n_2-1}$ for $1\le i_2\le 2^{n_1-1}$.
\end{proof}

\begin{rem}\label{n2=0}
	For $n_2=0$,  Proposition \ref{tyseq} also holds. In other words, 
	Proposition \ref{tyseq} is actually the generalization of  Lemma \ref{prop8}. 
\end{rem}

\subsection{Proof of Theorem \ref{cpthm1}}\

In this subsection, we devote to proving Theorem \ref{cpthm1}. For any $n\ge 2$, $n_1,n_2\ge1,\ n_1+n_2=n$, by  the definition of $\theta(n,k,t)$ and \eqref{big},
\begin{equation}\label{eq1}
	\sum_{i=1}^{2^{n_1-1}} \theta(n,f^{-1}(A_i))\ge
	\sum_{i=1}^{2^{n_1-1}}\theta(n,2^{n-1},s_i)\ge
	\sum_{i=1}^{2^{n_1-1}}\theta(n,2^{n-1},s_i^{(1)}),
\end{equation}
where $s_i^{(1)}$ is $s_i$ if $0\le s_i\le 2^{n-3}$ and $2^{n-3}$ otherwise.
It is clear that $(s_i^{(1)})_{i=1}^{2^{n_1-1}}$ satisfies condition (C1)(C2) and (C3), where\\
\textbf{(C3)}.all the entries are less than or equal to $2^{n-3}$.

For any $n\ge 2$, $n_1\ge 1, n_2\ge 0,\ n_1+n_2=n$. We say a sequence  $(x_{i})_{i=1}^{2^{n_1-1}}$ is of CONDITION $C_{n,n_2}$,\ if it satisfies  conditions (C1)(C2)(C3) corresponding to $n,n_2$. What is more,
\begin{lem}\label{halfseq}
For any $n\ge 2$, $n_1, n_2\ge 1,\ n_1+n_2=n$.
If  $(x_{i})_{i=1}^{2^{n_1-1}}$ is an even sequence with CONDITION $C_{n,n_2}$, then the sequence  $(x_{i}/2)_{i=1}^{2^{n_1-1}}$ is of CONDITION $C_{n-1,n_2-1}$. 
\end{lem}
It  is direct to check Lemma \ref{halfseq} holds due to the definition of CONDITION $C_{n,n_2}$.
Next, we are only interested in those subsets of $Q_n$ with small type.
Using \eqref{E(k)} repeatedly, we obtain the recurrence relation of $E(k)$ in Lemma \ref{recurrenceE(k)}. By this Lemma, we have
Lemma \ref{key} which  provides the relation of $\theta(n,2^{n-1},t)$  and $\theta(n,2^{n-1},t\pm 1)$ when $t$ is odd. It will play an important role in the rest proof.

\begin{lem}\label{recurrenceE(k)}
For	$k>0,\ k=\sum_{i=1}^{N} 2^{c_i},\ N\ge 1,\ 0\le c_1<c_2<\cdots<c_N=\left \lfloor  \log_2 k\right \rfloor$, then
\mbox{(i).} $E(2k+1)=E(2k)+N$,\ \ 
\mbox{(ii).} $E(2k+2)=E(2k+1)+N+1.$
\end{lem} 
\begin{proof}
(i). Since $2k=\sum_{i=1}^{N} 2^{c_i+1},\ 2k+1=2^0+\sum_{i=1}^{N} 2^{c_i+1}$, by \eqref{E(k)},
$$\begin{array}{cl}
&E(2k)=\sum_{i=1}^{N}(N-i)2^{c_i+1}+(c_i+1) 2^{c_i},\\
&E(2k+1)=N\cdot 2^0+0+\sum_{i=1}^{N}(N-i)2^{c_i+1}+(c_i+1) 2^{c_i}.		
\end{array}$$
As mentioned above, result (i) follows. 

(ii). \textbf{Case 1:} $c_1>0$. It implies that $2k+2=2^1+\sum_{i=1}^{N} 2^{c_i+1}$.
Consequently,  $$E(2k+2)=N\cdot 2^1+1+\sum_{i=1}^{N}(N-i)2^{c_i+1}+(c_i+1)2^{c_i}.$$
Comparing $E(2k+2)$ with $E(2k+1)$ in (i), result (ii) follows.

\textbf{Case 2:} $c_1=0$. There is $1\le N_1\le N$ such that $$k=\sum_{i=1}^{N_1} 2^{i-1}+\sum_{i=N_1+1}^{N} 2^{c_i},\ \ \mbox{for}\  c_{N_1+1}>N_1.$$
Then
$2k=\sum_{i=1}^{N_1} 2^{i}+\sum_{i=N_1+1}^{N} 2^{c_i+1}$, and
$$\begin{array}{rcl}
E(2k+1)&=&E(2k)+N\vspace{1.5ex}\\
&=&\sum_{i=1}^{N_1}((N-i)2^{i}+i2^{i-1})+\alpha+N\vspace{1.5ex}\\
&=&\sum_{i=1}^{N_1}(N-i)2^{i}+(1+\sum_{i=2}^{N_1}i2^{i-1})+N+\alpha\vspace{1.5ex}\\
&=&\sum_{i=1}^{N_1-1}(N-i)2^{i}+(N-N_1)2^{N_1}+
(1+\sum_{i=1}^{N_1-1}(i+1)2^{i})+N+\alpha\vspace{1.5ex}\\
&=&(N+1)\sum_{i=1}^{N_1-1}2^{i}+(N-N_1)2^{N_1}+1+N+\alpha\vspace{1.5ex}\\
&=&(N+1)(2^{N_1}-2)+(N-N_1)2^{N_1}+1+N+\alpha\vspace{1.5ex}\\
&=&(2N-N_1+1)2^{N_1}+\alpha-(N+1),
\end{array}$$
where $\alpha=\sum_{i=N_1+1}^{N}(N-i)2^{c_i+1}+(c_i+1)2^{c_i}.$
Moreover, we have $2k+2=2^{N_1+1}+\sum_{i=N_1+1}^{N} 2^{c_i+1}$, and
$$E(2k+2)=(N-N_1)2^{N_1+1}+(N_1+1)2^{N_1}+\alpha
=(2N-N_1+1)2^{N_1}+\alpha.$$

Comparing $E(2k+2)$ with $E(2k+1)$, result (ii) still follows.
Notice that $\sum_{i=N_1+1}^{N} 2^{c_i}=0$, $\alpha=0$ if $N_1=N$, then Lemma (ii) still holds.
Combining these results with \textbf{Case 1} and \textbf{Case 2},   Lemma (ii) is proved.
\end{proof}

For any $k>0$, by Lemma \ref{recurrenceE(k)}, we have
\begin{equation}\label{key1}
2E(2k+1)=E(2k)+E(2k+2)-1.
\end{equation}

\begin{lem}\label{key}
	For $ 0< k \le 2^{n-4}-1$ with $n\ge 4$,
	$$2\theta(n,2^{n-1},2k+1)=\theta(n,2^{n-1},2k)+\theta(n,2^{n-1},2k+2)+4.$$
\end{lem}
\begin{proof}
	By \eqref{small}, we only need to prove
	$$2\theta(n-2,2k+1)=\theta(n-2,2k)+\theta(n-2,2k+2)+2.$$
Together with \eqref{theta(k)} and \eqref{key1}, we have
	$$\begin{array}{cl}
		2\theta(n-2,2k+1)
		&=2[(n-2)(2k+1)-2E(2k+1)]\\
		&=(n-2)(2k+2k+2)-2(E(2k)+E(2k+2)-1)\\
		&=(n-2)\cdot 2k-2E(2k)+ (n-2)(2k+2)-2E(2k+2)+2\\
		&=\theta(n-2,2k)+\theta(n-2,2k+2)+2.
	\end{array}$$	
\end{proof}

Furthermore,  we want to find the relation between $\sum_{t}\theta(n,2^{n-1},t)$  and $\sum_{t}\theta(n,2^{n-1},t\pm 1)$ when $t$ is odd.
\begin{lem}\label{general}
	For any $n\ge 3,k\ge 1$, positive odd sequence $(x_i)_{i=1}^k$ is less than $2^{n-3}$,then
	$$\sum_{i=1}^{k}\theta(n,2^{n-1},x_i)\ge \min\{\sum_{i=1}^{k}\theta(n,2^{n-1},x_i-1),\sum_{i=1}^{k}\theta(n,2^{n-1},x_i+1) \}.$$	
\end{lem}
\begin{proof}
	By Lemma \ref{key}, we have
$$\begin{array}{cl}\label{odd2even}
\sum_{i=1}^{k}\theta(n,2^{n-1},x_i)
&\ge\sum_{i=1}^{k}\frac{1}{2}(\theta(n,2^{n-1},x_i-1)
+\theta(n,2^{n-1},x_i+1))\vspace{1.5ex}\\
&=\frac{1}{2}(\sum_{i=1}^{k}\theta(n,2^{n-1},x_i-1)
+\sum_{i=1}^{k}\theta(n,2^{n-1},x_i+1))\vspace{1.5ex}\\
&\ge \min\{\sum_{i=1}^{k}\theta(n,2^{n-1},x_i-1),\sum_{i=1}^{k}\theta(n,2^{n-1},x_i+1)\}.
\end{array}$$	
\end{proof}
	
\begin{exam}
	Letting $(x_i)_{i=1}^{10}=\{5,9,3,13,13,13,9,7,5,9\}$. 
	 It is direct to check that
$\sum_{i=1}^{10}\theta(7,2^6,x_i)=996$, and
$\sum_{i=1}^{10}\theta(7,2^6,x_i-1)=952$,
$\sum_{i=1}^{10}\theta(7,2^6,x_i+1)=1000$.
 So,$996\ge 952$. Lemma holds!
\end{exam}

In special, when the type sequence consists of  a series of  constitutive odd numbers with maximum $2^{n-3}-1$, we have 
\begin{lem}\label{sometou}
		For $ 0\le k_1 \le 2^{n-4}-1$ with $n\ge 4$, 
		$$\sum_{k=k_1}^{2^{n-4}-1}\theta(n,2^{n-1},2k+1)\ge \sum_{k=k_1}^{2^{n-4}-1}\theta(n,2^{n-1},2k+2).$$	
	\end{lem}
\begin{proof}
	By Lemma \ref{key}, we have
	$$\begin{array}{cl}
	&\quad\sum_{k=k_1}^{2^{n-4}-1}\theta(n,2^{n-1},2k+1)\vspace{1.5ex}\\
	&=\sum_{k=k_1}^{2^{n-4}-2}\theta(n,2^{n-1},2k+2)+\frac{1}{2}(\theta(n,2^{n-1},2k_1)+\theta(n,2^{n-1},2^{n-3}))\vspace{1.5ex}\\
	&\quad+2(2^{n-4}-1-k_1+1)\vspace{1.5ex}\\
	&=\sum_{k=k_1}^{2^{n-4}-1}\theta(n,2^{n-1},2k+2)+\frac{1}{2}(\theta(n,2^{n-1},2k_1)-\theta(n,2^{n-1},2^{n-3}))\vspace{1.5ex}\\
	&\quad+2(2^{n-4}-k_1)\vspace{1.5ex}\\
	&=\sum_{k=k_1}^{2^{n-4}-1}\theta(n,2^{n-1},2k+2)+
	(\theta(n-2,2k_1)-\theta(n-2,2^{n-3}))\vspace{1.5ex}\\
	&\quad+2^{n-3}-2k_1\vspace{1.5ex}\\
	&=\sum_{k=k_1}^{2^{n-4}-1}\theta(n,2^{n-1},2k+2)+\theta(n-3,2k_1),
	\end{array}$$
		where the third and last equality are due to formula \eqref{small}, and \eqref{recurred}, respectively.
		Since $2k_1<2^{n-3}$,\ $\theta(n-3,2k_1)\ge 0$ and the lemma follows.	
	\end{proof}

\begin{exam}
	It is direct to check that
	$\sum_{k=2}^{7}\theta(7,2^6,2k+1)=608\ge 600=\sum_{k=2}^{7}\theta(7,2^6,2k+2)$,
	although $\sum_{k=2}^{7}\theta(7,2^6,2k+2)=600> 592=\sum_{k=2}^{7}\theta(7,2^6,2k)$.
	  Lemma holds!
\end{exam}
If the type sequence consists of  a series of  constitutive odd numbers without  $2^{n-3}-1$, then it  also can be dealt with in the way used in Lemma \ref{general}.
With above preparations, we have the following proposition.
\begin{prop}\label{lemodd}
	For any $n\ge 2$, $n_1,n_2\ge1,\ n_1+n_2=n$.
	For any  integer sequence $(x_i)_{i=1}^{2^{n_1-1}}$ with CONDITION $C_{n,n_2}$, then there exists an  even sequence $(x_i^\ast)_{i=1}^{2^{n_1-1}}$ with CONDITION $C_{n,n_2}$
	 such that
\begin{equation}\label{deleteodd}
	\sum_{i=1}^{2^{n_1-1}}\theta(n,2^{n-1},x_i)\ge
	\sum_{i=1}^{2^{n_1-1}}\theta(n,2^{n-1},x_i^\ast).
	\end{equation}
\end{prop}
	
\begin{proof}
We discuss $(x_i)_{i=1}^{2^{n_1-1}}$ in following two cases.

\textbf{Case 1:} $n_2>1$. 
Construct a new sequence $(x_i^{(1)})_{i=1}^{2^{n_1-1}}$. 
Define $x_i^{(1)}$ to be $x_i$ if $x_i$ is even for each $1\le i \le 2^{n_1-1}$. Then,
observe all the odd part sequences between  two even in $(x_i)_{i=1}^{2^{n_1-1}}$. 
WLOG, only consider such an odd part sequence $(x_i)_{i=k_1}^{k_2}$ for $1\le k_1 \le k_2\le 2^{n_1-1}$, the other odds are same. 
For $k_1 \le i \le k_2$, define $x_i^{(1)}$ to be $x_i -1$ if
$$\sum_{i=k_1}^{k_2}\theta(n,2^{n-1},x_i-1)
\le \sum_{i=k_1}^{k_2}\theta(n,2^{n-1},x_i+1)$$  
and $x_i+1$ otherwise.
By Lemma \ref{general}, we have
\begin{equation}\label{case1}
	\sum_{i=1}^{2^{n_1-1}}\theta(n,2^{n-1},x_i)\ge \sum_{i=1}^{2^{n_1-1}}\theta(n,2^{n-1},x_i^{(1)}).
\end{equation}
It is direct to see that  $(x_i^{(1)})_{i=1}^{2^{n_1-1}}$ is an even sequence. 
Next, we study the properties of sequence $(x_i^{(1)})_{i=1}^{2^{n_1-1}}$.

\textbf{claim 1.1}. $(x_i^{(1)})_{i=1}^{2^{n_1-1}}$ is $n_2-$level\ circular\ continuous.

$\vartriangleleft$.  WLOG, assume $|x_1-x_2|\le 2^{n_2}$. 

 $\bullet\quad$For even $x_1$,$x_2$ or odd $x_1$,$x_2$. Then
$$|x_1^{(1)}-x_2^{(1)}|=|x_1-x_2|\le 2^{n_2}.$$ 

$\bullet\quad$For even $x_1$ and odd $x_2$. Then $|x_1-x_2|<2^{n_2}$,
$$|x_1^{(1)}-x_2^{(1)}|=|x_1-(x_2\pm1)|\le |x_1-x_2|+1\le2^{n_2}.$$
\hfill$\vartriangleright$

\textbf{claim 1.2}.There exist at least two $1\le i \le 2^{n_1-1}$ such that $x_i^{(1)}\ge2^{n-3}-2^{n_2-1}$.

$\vartriangleleft$. WLOG,  assume $x_1\ge 2^{n-3}-2^{n_2-1}$.

$\bullet\quad$For even $x_{1}\ge 2^{n-3}-2^{n_2-1}$. Then $$x_1^{(1)}=x_1\ge 2^{n-3}-2^{n_2-1}.$$

$\bullet\quad$For odd $x_{1}> 2^{n-3}-2^{n_2-1}$. Then
$$x_1^{(1)}\ge x_{1}-1\ge2^{n-3}-2^{n_2-1}.$$
\hfill$\vartriangleright$

\textbf{claim 1.3}. All the entries of $(x_i^{(1)})_{i=1}^{2^{n_1-1}}$ are less than or equal to $2^{n-3}$.

$\vartriangleleft$. WLOG, assume $x_1\le 2^{n-3}$. 

$\bullet\quad$For even $x_1\le 2^{n-3}$. Then
$$x_1^{(1)}=x_1\le 2^{n-3}.$$

$\bullet\quad$For odd $x_1 < 2^{n-3}$. Then
$$x_1^{(1)}\le x_1+1\le 2^{n-3}.$$
\hfill$\vartriangleright$

Summing up the discussions in claim 1.1--1.3, we conclude  that even sequence $(x_i^{(1)})_{i=1}^{2^{n_1-1}}$ is of CONDITION $C_{n,n_2}$.(See $(s_i^{(2)})_{i=1}^{16}$ in Table \ref{table1}). Together with \eqref{case1},
the result follows if $x_i^\ast=x_i^{(1)}$ for $1\le i \le 2^{n_1-1} $.
 
\textbf{Case 2:} $n_2=1$.
That is a sequence $(x_i)_{i=1}^{2^{n-2}}$ with CONDITION $C_{n,1}$.
It is worth mentioning that CONDITION $C_{n,1}$ is induced from the properties of sequence  $(t_{2i-1}^{(1)})_{i=1}^{2^{n-2}}$ or $(t_i^{(1)})_{i=1}^{2^{n-1}}$.
Recall that the sequence  $(t_i^{(1)})_{i=1}^{2^{n-1}}$ is of CONDITION $C_{n,0}$. Notice that if $t_2^{(1)}= 2^{n-3}$, then $t_1^{(1)},\ t_3^{(1)} \ge 2^{n-3}-1$.
Based on this fact, we consider the sequence $(x_i)_{i=1}^{2^{n-2}}$ with CONDITION $C_{n,1}$. If there are at least two elements equal to $2^{n-3}$, then we go to \emph{step 2} directly.
In the following,  we devote to working on the problem that
there is at most one element equal to $2^{n-3}$ in $(x_i)_{i=1}^{2^{n-2}}$.
If there is an unique element equal to $2^{n-3}$, then there must be one element, such as $x_j$, satisfying  an additional property $(\mathcal{U})$, where 

$(\mathcal{U})$. $x_j=2^{n-3}-1$, and
$\max\{x_{j-1}, x_{j+1}\}\ge 2^{n-3}-1$.

What's more, if all the elments  are strictly less than $2^{n-3}$ in $(x_i)_{i=1}^{2^{n-2}}$, then there must be two elements satisfiying property $(\mathcal{U})$.
Now we first consider the odd part sequences containing element satisfiying property $(\mathcal{U})$, then deal with the rest.

\emph{step 1.} Collect all the odd part sequences, each of which is strictly monotonous and contains element satisfiying property $(\mathcal{U})$ in $(x_i)_{i=1}^{2^{n-2}}$. Denote the union of such part sequences by $\Gamma$.
It's worth mentioning that we only pick those local {\it longest} odd part sequences.
WLOG, assume $(x_i)_{i=k_1}^{k_2}$ is strictly increasing for $1\le k_1 \le k_2 \le 2^{n-2}$, then $x_{k_2}=2^{n-3}-1$ and
$x_{k_1-1}\ge x_{k_1}$ if $x_{k_1-1}$ is odd.
Clearly,  we write $2^{n-2}$ instead of $k_1-1$ if $k_1=1$.(See Example \ref{eg}). 

Construct a new sequence $(x_i^{(1)})_{i=1}^{2^{n-2}}$. For $1\le i \le 2^{n-2}$, define $x_i^{(1)}$ to be $x_i+1$ if
$x_i \in \Gamma$ and $x_i$ otherwise.
 By  Lemma \ref{sometou}, we have
\begin{equation}\label{step 1}
	\sum_{i=1}^{2^{n-2}}\theta(n,2^{n-1},x_i)\ge \sum_{i=1}^{2^{n-2}}\theta(n,2^{n-1},x_i^{(1)}).
\end{equation}

Next, we study the properties of sequence $(x_i^{(1)})_{i=1}^{2^{n-2}}$.

\textbf{claim 2.1}. $(x_i^{(1)})_{i=1}^{2^{n-2}}$ is $1-$level\ circular\ continuous.

$\vartriangleleft$. WLOG, assume $|x_1-x_2|\le 2$. 

$\bullet\quad$For $x_1,x_2 \in \Gamma$ or $x_1,x_2 \notin \Gamma$. Then
$$|x_1^{(1)}-x_2^{(1)}|=|x_1-x_2|\le 2.$$  

$\bullet\quad$For $x_1 \notin \Gamma$, $x_2 \in \Gamma$. Then 
$$|x_1^{(1)}-x_2^{(1)}|=|x_1-(x_2+1)|=\left\{
\begin{array}{ccl}
&\le|x_1-x_2|+1=2,& \mbox{if }\	x_1\ \text{is\ even},\vspace{1.5ex}\\
&=|(x_1-x_2)-1|=1,& \mbox{if }\	x_1\ \text{is\ odd}.
\end{array}
\right.$$
Notice that $x_1\ge x_2$ when $x_1$ is odd.
If $x_2=2^{n-3}-1$, then  $x_1=2^{n-3}-1$ by property $(\mathcal{U})$, thus $x_1\in \Gamma$.
So, $x_2<2^{n-3}-1$, and $x_1\ge x_2$ by the statement of {\it longest}.

\hfill$\vartriangleright$

\textbf{claim 2.2}.There exist at least two $1\le i \le 2^{n-2}$ such that $x_i^{(1)}\ge2^{n-3}$.

$\vartriangleleft$.   WLOG, assume $x_1\ge2^{n-3}-1$ .

$\bullet\quad$For even $x_1>2^{n-3}-1$. Then 
$$x_1^{(1)}=x_1\ge2^{n-3}.$$

$\bullet\quad$For odd $x_1\ge2^{n-3}-1$. Then
$$x_1^{(1)}=x_1+1\ge2^{n-3}.$$
\hfill$\vartriangleright$

\textbf{claim 2.3}. All the entries of $(x_i^{(1)})_{i=1}^{2^{n-2}}$ are less than or equal to $2^{n-3}$.

$\vartriangleleft$. WLOG, assume $x_1\le 2^{n-3}$. 

$\bullet\quad$For $x_1\notin \Gamma$. Then
$$x_1^{(1)}=x_1\le 2^{n-3}.$$

$\bullet\quad$For $x_1\in \Gamma$. Then odd $x_1<2^{n-3}$,
$$x_1^{(1)}=x_1+1 \le 2^{n-3}.$$
\hfill$\vartriangleright$

Summing up the discussions in claim 2.1--2.3, we conclude  that sequence $(x_i^{(1)})_{i=1}^{2^{n-2}}$  is of CONDITION $C_{n,1}$ and contains at least two $2^{n-3}$.(See $(s_i^{(4)})_{i=1}^{16}$ in Table \ref{table1}).

\emph{step 2.} Pick all the odd part sequences between  two even in $(x_i^{(1)})_{i=1}^{2^{n-2}}$.
Following the same procedure in \textbf{Case 1}, an even sequence $(x_i^{(2)})_{i=1}^{2^{n-2}}$ is obtained.
 By  Lemma \ref{general}, we have
\begin{equation}\label{step 2}
	\sum_{i=1}^{2^{n-2}}\theta(n,2^{n-1},x_i^{(1)})\ge \sum_{i=1}^{2^{n-2}}\theta(n,2^{n-1},x_i^{(2)}).
\end{equation}

 By analogous proof of claim 1.1--1.3, we can show that even sequence $(x_i^{(2)})_{i=1}^{2^{n-2}}$ is of  CONDITION $C_{n,1}$. (See $(s_i^{(5)})_{i=1}^{16}$ in Table \ref{table1}).
 For each $1\le i \le 2^{n-2}$, letting $x_i^\ast=x_i^{(2)}$, then inequality \eqref{deleteodd} holds by \eqref{step 1} \eqref{step 2}. 
Combining these results with \textbf{Case 1} and \textbf{Case 2}, the proposition follows.  
\end{proof}

\begin{exam}\label{eg}
	For $n=7$, given a sequence $\{\ast \ast, 9,9,11,13,15,\ast \ast \}$ or $\{11,13,15,\ast \ast,9,9 \}$, then pick the longest part $\{9,11,13,15\}$ in step 1. If pick $\{11,13,15\}$, then $\{9,12,14,16\}$ doesnot maintain 1-level continous in step 1.
\end{exam}	

Now, we go back to the proof of Theorem \ref{cpthm1}.
\begin{proof}
It is seen in \eqref{eq1} that $(s_i^{(1)})_{i=1}^{2^{n_1-1}}$ is of CONDITION $C_{n,n_2}$.
By Proposition \ref{lemodd} and \eqref{double2}, there there exists an even sequence $(s_i^{(2)})_{i=1}^{2^{n_1-1}}$ with CONDITION $C_{n,n_2}$ such that
$$\sum_{i=1}^{2^{n_1-1}}\theta(n,2^{n-1},s_i^{(1)})\ge
\sum_{i=1}^{2^{n_1-1}}\theta(n,2^{n-1},s_i^{(2)})
=\sum_{i=1}^{2^{n_1-1}}2\theta(n-1,2^{n-2},s_i^{(3)}),$$
where $s_i^{(3)}=s_i^{(2)}/2$ for $1\le i \le 2^{n_1-1}$ . By Lemma \ref{halfseq},
the sequence $(s_i^{(3)})_{i=1}^{2^{n_1-1}}$ is of CONDITION $C_{n-1,n_2-1}$.

Using Proposition \ref{lemodd}, \eqref{double2} and Lemma \ref{halfseq} repeatedly, there exists a sequence $(s_i^{\ast})_{i=1}^{2^{n_1-1}}$ with CONDITION $C_{n_1,0}$, such that
$$\sum_{i=1}^{2^{n_1-1}}\theta(n,2^{n-1},s_i^{(1)})\ge
2^{n_2}\sum_{i=1}^{2^{n_1-1}}\theta(n_1,2^{n_1-1},s_i^{\ast}).$$

It is known in \cite{Guu,Liu2021} that any sequence $(t_i^{(1)})_{i=1}^{2^{n-1}}$ with CONDITION $C_{n,0}$,  
\begin{equation*}\label{cycle}
	\sum_{i=1}^{2^{n-1}}\theta(n,2^{n-1},t_i^{(1)}) \ge \sum_{i=1}^{2^{n-1}}\theta(n,2^{n-1},g_i^{n}). 
\end{equation*}
Thus,
$$\sum_{i=1}^{2^{n_1-1}}\theta(n,2^{n-1},s_i^{(1)})\ge
2^{n_2}\sum_{i=1}^{2^{n_1-1}}\theta(n_1,2^{n_1-1},s_i^{\ast})\ge 2^{n_2}\sum_{i=1}^{2^{n_1-1}}\theta(n_1,2^{n_1-1},g_i^{n_1})
$$
Together with \eqref{eq1} and \eqref{gi-n2}, we have
$$\sum_{i=1}^{2^{n_1-1}} \theta(n,f^{-1}(A_i))
\ge
\sum_{i=1}^{2^{n_1-1}} \theta(n,\xi_n^{-1}(A_i)).$$
Hence Theorem \ref{cpthm1} is proved!
\end{proof}

In the following, we will present an example to explain Theorem \ref{cpthm1} more clearly.
\begin{exam}\label{examf}
Take an embedding $f:Q_7\rightarrow C_{2^5}\times P_{2^{2}}$ with $(f(v))_{v\in\{0,1\}^7}=$
$$\begin{array}{l}
(3, 2, 6, 7, 5, 4, 12, 13, 11, 9, 8, 24, 25, 27, 26, 30,31,29, 28, 20, 21,\\ 
23, 22, 18, 19, 17, 16, 48, 49, 51, 50, 54, 55, 53, 52, 60, 61, 63, 62,\\ 
58, 59, 57, 56, 40, 41, 43, 42, 46, 47, 45, 44, 36, 37, 39, 38, 34, 35, \\
33, 32, 96, 97,99, 98, 102, 103, 101, 100, 108, 109, 111, 110, 106,\\
 107,105, 104, 120, 121, 123, 15, 14, 10, 122,  126, 127, 125, 124,\\ 
116, 117,119, 118, 114, 115, 113, 112, 80, 81, 83, 82, 86, 87, 85,\\ 84,92, 93, 95, 
94, 90, 91, 89, 88, 72, 73, 75, 74, 78, 79,77, 76,\\ 68, 69, 71, 70, 66, 67, 65, 64, 0, 1),
\end{array}$$
where $f(v)$ are list in lexicographic order of $v\in\{0,1\}^7$. Based on this embedding,
for $1\le i \le 2^{4}$,  $s_i$ and $s_i^{(j)}$  are illustrated in Table \ref{table1}.
\end{exam}
\begin{table}[h]
	\centering
	\caption{Illustration of $s_i$ and $s_i^{(j)}$}
\begin{tabular}{c cccc cccc cccc cccc}
	\hline
	$i$&   1& 2& 3& 4&   5& 6&  7& 8& 9& 10& 11& 12& 13& 14& 15& 16\\ \hline
	$s_i$&        5& 9& 13& 15& 13& 10& 6& 3& 5& 9& 13& 17& 14& 10& 6&  3\\	
	$s_i^{(1)}$&  5& 9& 13& 15& 13& 10& 6& 3& 5& 9& 13& 16& 14& 10& 6&  3\\	
	$s_i^{(2)}$&  4& 8& 12& 14& 12& 10& 6& 2& 4& 8& 12& 16& 14& 10& 6&  2\\	
	$s_i^{(3)}$&  2& 4& 6&   7& 6&  5&  3& 1& 2& 4& 6&  8&  7&   5& 3&  1\\		
	$s_i^{(4)}$&  2& 4& 6&   7& 6&  5&  3& 1& 2& 4& 6&  8&  8&   6& 4&  2\\
	$s_i^{(5)}$&  2& 4& 6&   8& 6&  4&  2& 0& 2& 4& 6&  8&  8&   6& 4&  2\\
	$s_i^{(6)}$&  1& 2& 3&   4& 3&  2&  1& 0& 1& 2& 3&  4&  4&   3&  2&  1\\	
	\hline
\end{tabular}
	\label{table1}	
\end{table}

Observe Table \ref{table1}, $n=7, n_1=5, n_2=2$.
 
Firstly, since $2^{n-3}=16$, let $s_{12}^{(1)}=16$ instead of $s_{12}=17$. Others stay still.
It is direct to check that the sequence $(s_i^{(1)})_{i=1}^{16}$ is of
CONDITION $C_{7,2}$. 
Following the procedure in \textbf{Case 1}  of Proposition \ref{lemodd}, 
$$\begin{array}{cl}
	&\theta(7,2^{6},s_{16}^{(1)}-1)+
	\sum_{i=1}^{5}\theta(7,2^{6},s_i^{(1)}-1)=576\vspace{1.5ex}\\
	<&592=\theta(7,2^{6},s_{16}^{(1)}+1)+
	\sum_{i=1}^{5}\theta(7,2^{6},s_i^{(1)}
	+1),
\end{array}
$$ $$\sum_{i=8}^{11}\theta(7,2^{6},s_i^{(1)}-1)=368<392=
\sum_{i=8}^{11}\theta(7,2^{6},s_i^{(1)}
+1),$$
then
$s_i^{(2)}=s_i^{(1)}-1$ for $i\in\{1,2,3,4,5,8,9,10,11,16\}$. Others stay still.
Thus,
even sequence $(s_i^{(2)})_{i=1}^{16}$ is of CONDITION $C_{7,2}$.
\begin{equation*}
\sum_{i=1}^{2^4} \theta(7,f^{-1}(A_i))\ge
\sum_{i=1}^{2^4}\theta(7,2^6,s_i)\ge
\sum_{i=1}^{2^4}\theta(7,2^6,s_i^{(1)})\ge
\sum_{i=1}^{2^4}\theta(7,2^6,s_i^{(2)}).
\end{equation*}

Secondly,
for $1\le i\le 2^4$, letting $s_i^{(3)}=s_i^{(2)}/2$, then sequence $(s_i^{(3)})_{i=1}^{16}$ is of CONDITION $C_{6,1}$ by Lemma \ref{halfseq}. 
Follow the procedure in \textbf{Case 2}  of Proposition \ref{lemodd}.
Step 1, since 
$s_{13}^{(3)}$ satisfies property $\mathcal{U}$, then
$s_i^{(4)}=s_i^{(3)}+1$ for $i\in\{13,14,15,16\}$. Others stay still. 
Step 2,  since $\theta(6,2^5,s_4^{(4)}-1)=52>48=\theta(6,2^5,s_4^{(4)}+1)$, then $s_4^{(5)}=s_4^{(4)}+1$. 
Similarly, $s_i^{(5)}=s_i^{(4)}-1$ for $i\in\{6,7,8\}$.
Others stay still. 
Thus,
even sequence $(s_i^{(5)})_{i=1}^{16}$ is of CONDITION $C_{6,1}$.
\begin{equation*}
\sum_{i=1}^{2^4}\theta(7,2^6,s_i^{(2)})=
2\sum_{i=1}^{2^4}\theta(6,2^5,s_i^{(3)})
\ge 2\sum_{i=1}^{2^4}\theta(6,2^5,s_i^{(5)}).
\end{equation*}

Finally, for $1\le i\le 2^4$, letting $s_i^{(6)}=s_i^{(5)}/2$, then sequence $(s_i^{(6)})_{i=1}^{16}$ is of CONDITION $C_{5,0}$ by Lemma \ref{halfseq}. 
\begin{equation*}
2\sum_{i=1}^{2^4}\theta(6,2^5,s_i^{(5)})
= 2^2\sum_{i=1}^{2^4}\theta(5,2^4,s_i^{(6)})
\ge
2^2\sum_{i=1}^{2^4}\theta(5,2^4,g_i^{5})=
\sum_{i=1}^{2^4}\theta(7,\xi_7^{-1}(A_i)).
\end{equation*}
As mentioned above, we have
$$\sum_{i=1}^{2^4} \theta(7,f^{-1}(A_i))\ge
\sum_{i=1}^{2^4}\theta(7,\xi_7^{-1}(A_i)).$$
\section{Conclusions}\

We identify an exact layout of hypercube into a cylinder. We introduce new technique related to EIP to estimate the minimum EWHC. This technique is based on CONDITION $C_{n,n_2}$, which exactly describes the properties of type sequence corresponding to some subsets of hypercube. Following this idea, the problem of embedding wirelength of hypercube into torus and related graphs is also considered\cite{Tang2021b}.

\textbf{Acknowledgements}  
The author is grateful to Prof. Qinghui Liu for his thorough review and suggestions.
This work is supported by the National Natural Science Foundation of China, No.11871098.

\end{document}